\documentclass[10pt]{amsproc}
\usepackage{enumerate}
\usepackage{graphicx}
\usepackage{amssymb}
\usepackage{epstopdf}
\usepackage{amsmath,amsthm,amscd,amssymb}
\usepackage{latexsym}
\usepackage[colorlinks,citecolor=red,pagebackref,hypertexnames=false]{hyperref}
\usepackage{geometry}                
\numberwithin{equation}{section}

\theoremstyle{plain}
\newtheorem{theorem}{Theorem}[section]
\newtheorem{lemma}[theorem]{Lemma}
\newtheorem{corollary}[theorem]{Corollary}

\theoremstyle{definition}

\newtheorem{case[theorem]}{Case}

\theoremstyle{remark}

\numberwithin{equation}{section}

\def\dH{\dim_{{\mathcal H}}}
\def\R{\Bbb R}

\def\e{\epsilon}

\begin{document}

\title{Group actions, the Mattila integral and applications} 


\author{Bochen Liu}

\date{today}


\email{bochen.liu1989@gmail.com}
\address{Department of Mathematics, Bar-Ilan University, Ramat Gan, Israel}


\maketitle
\begin{abstract}
The Mattila integral, 
$$ {\mathcal M}(\mu)=\int {\left( \int_{S^{d-1}} {|\widehat{\mu}(r \omega)|}^2 d\omega \right)}^2 r^{d-1} dr,$$
developed by Mattila, is the main tool in the study of the Falconer distance problem. In this paper, with a very simple argument, we develop a generalized version of the Mattila integral. Our first application is to consider the product of distances $$(\Delta(E))^k= \left\{\prod_{j=1}^k |x^j-y^j|: x^j, y^j\in E\right\} $$
and show that when $d\geq 2$,
$(\Delta(E))^k$ has positive Lebesgue measure if $\dH(E)>\frac{d}{2}+\frac{1}{4k-1}$. Another application is, we prove for any $E,F,H\subset\R^2$, $\dH(E)+\dH(F)+\dH(H)>4$, the set
$$E\cdot(F+H)=\{x\cdot(y+z): x\in E, y\in F, z\in H\}$$
has positive Lebesgue.  

\end{abstract}
\section{Introduction}
\subsection{Erd\H{o}s-Falconer problems}
One of the most interesting and far reaching problems of modern geometric measure theory is the Falconer distance problem, which asks how large the Hausdorff dimension of a compact set $E \subset {\Bbb R}^d$, $d \ge 2$, (henceforth denoted by $\dH(E)$) needs to be in order for the set of distances 
$$ \Delta(E)=\{|x-y|: x,y \in E \}$$ to have positive Lebesgue measure. The best currently known results are due to Wolff  (\cite{Wol99}) in two dimensions and Erdogan (\cite{Erd05}) in higher dimensions. They proved that $|\Delta(E)|>0$ if the Hausdorff dimension of $E$ is greater than $\frac{d}{2}+\frac{1}{3}$. On the other hand, Bourgain (\cite{Bou03}) showed that for planar sets $E\subset\R^2$, $\dH(\Delta(E))>\frac{1}{2}+c_0$ for some absolute constant $c_0>0$ whenever $\dH(E)\geq 1$. The conjectured exponent is $\frac{d}{2}$ and it was shown by Falconer in \cite{Fal85} that this exponent would be best possible.

Both Wolff and Erdogan used the paradigm to attack the Falconer distance problem invented by Mattila in \cite{Mat87}. That is, to show $\Delta(E)$ has positive Lebesgue measure, it suffices to show that there exists a Borel measure $\mu$ on $E$ such that
\begin{equation} \label{mattilaintegral} {\mathcal M}(\mu)=\int {\left( \int_{S^{d-1}} {|\widehat{\mu}(r \omega)|}^2 d\omega \right)}^2 r^{d-1} dr<\infty.\end{equation} We call $\mathcal M(\mu)$ the Mattila integral of $\mu$. While ${\mathcal M}(\mu)$ and its connection with the Falconer problem can be derived directly, as is done in \cite{Mat87} and \cite{Wol99}, authors in \cite{GILP15} take the following geometric point of view that has been proved so useful in the solution of the Erd\H{o}s distance conjecture in the plane by Guth and Katz (\cite{GK15}).

Notice that $|x-y|=|x'-y'|$ if and only if there exists $\theta\in O(d)$ such that $x-\theta x'=y-\theta y'$. So the orthogonal group $O(d)$ is the invariant group of the distance problem. The key observation in \cite{GILP15} is, the Mattila integral can be written as an integral with the Haar measure $\lambda_{O(d)}$ on $O(d)$ involved, i.e.
$$\int_0^\infty \left(\int_{S^{d-1}}|\hat{\mu}(r\omega)|^2\,d\omega\right)^2\,r^{d-1}\,dr= c_d\int |\hat{\mu}(\xi)|^2\left(\int_{O(d)}|\hat{\mu}(\theta\xi)|^2\,d\lambda_{O(d)}(\theta)\right)\,d\xi. $$
With this new observation of the Mattila integral, author in \cite{GILP15} obtained a generalized version of the Mattila integral to study when the set of $k$-simplices,
$$T_k(E)=\{\left(\dots,|x^i-x^j|,\dots \right)\in \R^{k+1 \choose 2}: x^i\in E\subset\R^d, 1\leq i<j\leq k+1\}$$
has positive Lebesgue measure.


Although different derivations of the Mattila integral \eqref{mattilaintegral} are given, none of them is easy to generalize. Proofs in \cite{Mat87}, \cite{Wol03} use the asymptotic expansion of Bessel functions, while the proof in \cite{GILP15} relies on counting co-dimension of intersections of sub-manifolds in $O(d)$. In this paper, we give a much simpler argument to develop a more general result.
\vskip.125in
Suppose $E_1,\dots,E_{k+1}\subset\R^d$ are compact sets and $\Phi:\R^{d(k+1)}\rightarrow \R^m$ is a Lipschitz map whose $m$-dimensional Jacobian $J_m\Phi$ never vanishes on $E_1\times\cdots\times E_{k+1}$. Suppose there exists a locally compact topological group $G$ acting continuously on $\R^d$ such that $$\Phi(x^1,x^2,\dots,x^{k+1})=\Phi(y^1,y^2,\dots,y^{k+1})$$ if and only if $$(y^1,y^2,\dots,y^{k+1})=(gx^1,gx^2,\dots,gx^{k+1})$$ for some $g\in G$. Denote
$$\Delta_\Phi(E_1,\dots,E_{k+1}):= \{\Phi(x^1,\dots,x^{k+1}):x^j\in E_j \}.$$
We shall investigate how large the Hausdorff dimension of $E_j$ needs to be to ensure that $\Delta_\Phi(E_1,\dots,E_{k+1})$ has positive $m$-dimensional Lebesgue measure.
\vskip.125in
We need more notations. Let $\phi\subset C_0^\infty$, $\int\phi=1$ and $\phi^\epsilon=\frac{1}{\epsilon^d}\phi(\frac{\cdot}{\epsilon})$. Suppose $\mu_j$ is a probability measure on $E_j$ and denote $\mu_j^\epsilon = \mu_j*\phi^\epsilon\in C_0^\infty(\R^d)$. Then one can define a probability measure $\nu^\epsilon$ on the $\epsilon$-neighborhood of $\Delta_\Phi(E_1,\dots,E_{k+1})\subset\R^m$ by
\begin{equation}\label{defmeasure} \int f(\vec{t})\, d\nu^\epsilon(\vec{t})=\int_{\R^{d(k+1)}} f(\Phi(x^1,\dots,x^{k+1})) \,\mu_1^\epsilon(x^1)\cdots \mu_{k+1}^\epsilon(x^{k+1})\,dx^1\cdots dx^{k+1}.
\end{equation}
In other words, $\nu^\epsilon$ is the push-forward of $\mu_1^\epsilon\times\cdots\times\mu_{k+1}^\epsilon$ under $\Phi$. Our main theorem is the following.

\begin{theorem}
	\label{main}
	Suppose $\lambda_G$ is a right Haar measure on $G$. With notations above,
	\begin{equation}\label{gengenMat}
		\int |\widehat{\nu^\epsilon}(\xi)|^2\,d\xi \approx \int_{G}\prod_{j=1}^{k+1}\left(\int_{\R^d} \mu_j^\e(x)\,\mu_j^\e(gx)\,dx\right) d\lambda_G(g),
	\end{equation}
where the implicit constant is independent in $\epsilon$. Moreover, if the right hand side is bounded above uniformly in $\epsilon$, the set $\Delta_\Phi(E_1,\dots,E_{k+1})\subset\R^m$ has positive Lebesgue measure.
\end{theorem}

In particular, if, like distances, $k=2$, $E_1=E_2=E$ and $\Phi$ is translation invariant, which means
$\Phi(x,y)= \Phi(x',y') $
if and only if there exists $z\in\R^d$, $g\in G$ such that
$$x'=gx+z,\ y'=gy+z.  $$
The right hand side of \eqref{gengenMat} becomes
$$\int_{G}\int_{\R^d}\left|\int_{\R^d} \mu^\e(x)\,\mu^\e(z-g x)\,dx\right|^{2}\,dz\,d\lambda_{G}(g).$$
By Plancherel in $x$ it equals
$$\int_{G}\int_{\R^d}\left|\int_{\R^d} \widehat{\mu^\e}(\xi)\,e^{2\pi i z\cdot\xi}\widehat{\mu^\e}(g \xi)\,d\xi\right|^{2}\,dz\,d\lambda_{G}(g).$$
Then by Plancherel in $z$, as $\e\rightarrow 0$, it equals
$$\int |\hat{\mu}(\xi)|^2\left(\int_{G}|\hat{\mu}(g\xi)|^2\,d\lambda_{G}(g)\right)\,d\xi.$$
Thus we have the following corollary.
\begin{corollary}\label{CorMat}
	Given $E\subset\R^d$, $d\geq 2$ and $\Phi:\R^d\times\R^d\rightarrow\R^m$ is a Lipschitz map with non-vanishing Jacobian $J_m\Phi$ on $E\times E$. Suppose there exists a locally compact topological group $G$ such that $$\Phi(x,y)= \Phi(x',y') $$
if and only if there exists $z\in\R^d$, $g\in G$ such that
$$x'=gx+z,\ y'=gy+z.$$
Then 
$$\{\Phi(x,y):x,y\in E\} $$
has positive Lebesgue measure if there exists a Borel measure $\mu$ on $E$ and Haar measure $\lambda_{G}$ on $G$ such that
$$\int |\hat{\mu}(\xi)|^2\left(\int_{G}|\hat{\mu}(g\xi)|^2\,d\lambda_{G}(g)\right)\,d\xi<\infty.$$
\end{corollary}

\subsection{Product of distance sets}
As applications of Theorem \ref{main}, we first consider the product of distance sets,
$$(\Delta(E))^k= \left\{\prod_{j=1}^k |x^j-y^j|: x^j, y^j\in E\right\}. $$

\begin{theorem}\label{d>1}
Suppose $E\subset\R^d$, $d\geq 2$, $k\in\mathbb{Z}^+$. Then $(\Delta(E))^k$ has positive Lebesgue measure if $\dH(E)>\frac{d}{2}+\frac{1}{4k-1}$.
\end{theorem}

This result is not trivial. Notice that there exists $A\subset\R$ of Hausdorff dimension $1$ while $|A^k|=0$ for any $k\in\mathbb{Z}^+$. Also whether $|\Delta(E)|>0$ is still unknown when $\dH(E)\leq\frac{d}{2}+\frac{1}{3}$. 
\subsection{Continuous sum-product problems}
Given any finite set $A\subset\mathbb{N}$, one can define its sum set and product set by
$$A+A=\{a_1+a_2:a_1,a_2\in A\},$$
$$AA=\{a_1a_2:a_1,a_2\in A\},  $$
respectively. The Erd\H{o}s-Szemer\'edi conjecture (\cite{ES83}) states that for any $\epsilon>0$,
$$\max\{\#(A+A), \#(AA)\}\ge C_\epsilon \#(A)^{2-\epsilon}.$$
There are various formulations of sum-product estimates, also in different settings, such as finite fields. For more information one can see \cite{KS16}, \cite{MRS15}, \cite{MRS17}, \cite{BKT04}, \cite{MPRRS17} and references therein.

In this paper, by applying Theorem \ref{main}, we obtain the following continuous analog of sum-product estimate. Denote $x\cdot y$ as the dot-product between $x,y\in\R^2$.
\begin{theorem}\label{sum-prod}
	Suppose $E, F, H\subset\R^2$. With notations above,
	$$|E\cdot(F+H)|>0$$
	whenever $\dH(E)+\dH(F)+\dH(H)>4$. This dimensional threshold is generally sharp.
\end{theorem}
To see the sharpness, one can take $E=\{(0,0)\}$, $F=H=[0,1]^2$.

Taking $E=A\times\{0\}$, $F=B\times [0,1]$, $H=C\times [0,1]$, Theorem \ref{sum-prod} implies the following sum-product estimate on $\R$.
\begin{corollary}\label{cor1}
Suppose $A, B, C\subset\R$ and $\dH(A)+\dH(B)+\dH(C)>2$, then
$$|A(B+C)|>0.$$
This dimensional threshold is generally sharp.
\end{corollary}
Sharpness follows from $A=\{0\}$, $B=C=[0,1]$.

As a remark, we point out that the dimensional exponent $4$ in Theorem \ref{sum-prod} can also be obtained from Peres-Schlag's generalized projection theorem (\cite{PS00}). However, our proof is straightforward and shows the power of the group action method.




{\bf Notations.} Throughout this paper, $X\lesssim Y$ means $X\leq CY$ for some constant $C>0$. $X\lesssim_\epsilon Y$ means $X\leq C_\epsilon Y$ for some constant $C_\epsilon>0$, depending on $\epsilon$.
\vskip.125in
{\bf Acknowledgements.} The author really appreciates Prof. Ka-Sing Lau from Chinese University of Hong Kong for his financial support of research assistantship. 

\section{Proof of Theorem \ref{main}}
\subsection{Idea of the proof}
We first sketch the idea of the proof. A rigorous proof comes later in this section.
\vskip.125in
Denote $\mu_j$ as a probability Frostman measure on $E_j$. By the assumption on $\Phi$, 
$$\Phi^{-1}(\vec{t})=\{(x^1,\dots,x^{k+1}): \Phi(x^1,\dots,x^{k+1})=\vec{t}\}=\{(gx_{t}^1,\dots,gx_{t}^{k+1}):g\in G \},  $$
where $(x_t^1,\dots,x_t^{k+1})\in \Phi^{-1}(\vec{t})$ is arbitrary but fixed. Roughly speaking, we have two ways to define a measure $\nu$ on $\Delta_\Phi(E_1,\dots,E_{k+1})$,
$$ \nu(\vec{t})=\int_{\Phi^{-1}(\vec{t})} \,\mu_1(x^1)\cdots \mu_{k+1}(x^{k+1})\,d\mathcal{H}^{d(k+1)-m}(x^1,\dots,x^{k+1})$$
and
$$\nu(\vec{t})\approx\int_G \,\mu_1(gx_t^1)\cdots \mu_{k+1}(gx_t^{k+1})\,d\lambda_G(g),$$
where $(x_t^1,\dots,x_t^{k+1})\in\Phi^{-1}(\vec{t})$ is arbitrary but fixed.

Multiplying these two expressions, $||\nu||_{L^2}^2$ is approximately equal to
\begin{equation*}
\int\int_{\Phi^{-1}(\vec{t})} \,\mu_1(x^1)\cdots \mu_{k+1}(x^{k+1})\left(\int_G \,\mu_1^\epsilon(gx_t^1)\cdots \mu_{k+1}^\epsilon(gx_t^{k+1})\,d\lambda_G(g)\right)d\mathcal{H}^{d(k+1)-m}(x^1,\dots,x^{k+1})\,d\vec{t}.
\end{equation*}
	By the invariance of the Haar measure, we may replace $(x_t^1,\dots,x_t^{k+1})$ by $(x^1,\dots,x^{k+1})$. Also it is not hard to see that
$$d\mathcal{H}^{d(k+1)-m}\big|_{\Phi^{-1}(\vec{t})}(x^1,\dots,x^{k+1})\,d\vec{t}\approx dx^1\dots dx^{k+1}.  $$
Hence 
\begin{equation*}
	\begin{aligned}
		||\nu||_{L^2}^2\approx&\int_{\R^{d(k+1)}} \,\mu_1(x^1)\cdots \mu_{k+1}(x^{k+1})\,\left(\int_G \,\mu_1(gx^1)\cdots \mu_{k+1}(gx^{k+1})\,d\lambda_G(g)\right)dx\\=&\int_{G}\prod_{j=1}^{k+1}\left(\int_{\R^d} \mu_j(x)\,\mu_j(gx)\,dx\right)\,d\lambda_G(g).
	\end{aligned}
\end{equation*}

\subsection{Rigorous proof} We need the coarea formula to prove the theorem. For smooth cases the coarea formula follows from a simple change of variables. More general forms of the formula for Lipschitz functions were first established by Federer in 1959 and later generalized by different authors. For references, one can see \cite{Fed69}. We will use the following version to prove Theorem \ref{main}.

\begin{theorem}
	[Coarea formula, 1960s]
Let $\Phi$ be a Lipschitz function defined in a domain $\Omega\subset\R^{d(k+1)}$, taking on values in $\R^m$ where $m<d(k+1)$. Then for any $f\in L^1(\R^{d(k+1)})$,
$$\int_{\Omega} f(x)|J_m\Phi(x)|\,dx=\int_{\R^{m}}\left(\int_{\Phi^{-1}(\vec{t})} f(x)\,d\mathcal{H}^{d(k+1)-m}(x)\right)\,d\vec{t},  $$
where $J_k\Phi$ is the $m$-dimensional Jacobian of $\Phi$ and $\mathcal{H}^{d(k+1)-m}$ is the $(d(k+1)-m)$-dimensional Hausdorff measure.
\end{theorem}

With the coarea formula, \eqref{defmeasure} can be written as
\begin{equation*} 
\begin{aligned}
\int f(\vec{t})\, d\nu^\epsilon(\vec{t})=&\int_{\R^{d(k+1)}} f(\Phi(x^1,\dots,x^{k+1})) \,\mu_1^\epsilon(x^1)\cdots \mu_{k+1}^\epsilon(x^{k+1})\,dx^1\cdots dx^{k+1}\\=&\int_{\R^{m}} f(\vec{t})\left(\int_{\Phi^{-1}(\vec{t})} \,\mu_1^\epsilon(x^1)\cdots \mu_{k+1}^\epsilon(x^{k+1})\,\frac{1}{|J_m\Phi(x)|}\,d\mathcal{H}^{d(k+1)-m}(x)\right)d\vec{t}.
\end{aligned}
\end{equation*}

It follows that 
\begin{equation}\label{nu1}
\nu^\epsilon(\vec{t})=\int_{\Phi^{-1}(\vec{t})} \,\mu_1^\epsilon(x^1)\cdots \mu_{k+1}^\epsilon(x^{k+1})\,\frac{1}{|J_m\Phi(x)|}\,d\mathcal{H}^{d(k+1)-m}(x).
\end{equation}
On the other hand, since $\Phi$ is $G$-invariant, any Haar measure $\lambda_G$ on $G$ induces a measure $\sigma_t$ on $\Phi^{-1}(\vec{t})$ by
$$\int_{\Phi^{-1}(\vec{t})} f(x^1,\dots,x^{k+1})\, d\sigma_t = \int_{G} f(g x_t^1,\dots,g x_t^{k+1})\, d\lambda_G(g), $$
where $(x_t^1,\dots,x_t^{k+1})$ is any fixed point in $\Phi^{-1}(\vec{t})$. By the invariance of the Haar measure, $\sigma_t$ does not depedent on the choice of $(x_t^1,\dots,x_t^{k+1})$, 
and it must be absolutely continuous with respect to 
$\mathcal{H}^{d(k+1)-m}|_{\Phi^{-1}(\vec{t})}$, i.e., there exist a positive  function $\psi$ on $\Phi^{-1}(\vec{t})$ such that
$$\sigma_t=\psi\, \mathcal{H}^{d(k+1)-m}|_{\Phi^{-1}(\vec{t})}. $$

On any compact set, $\psi\approx 1$, so another expression of $\nu^\e$ follows,
\begin{equation}\label{nu2}
\begin{aligned}
\nu^\epsilon(\vec{t})=&\int_{\Phi^{-1}(\vec{t})} \,\mu_1^\epsilon(x^1)\cdots \mu_{k+1}^\epsilon(x^{k+1})\,\frac{1}{|J_m\Phi(x)|}\,d\mathcal{H}^{d(k+1)-m}(x)\\\approx&\int_G \,\mu_1^\epsilon(gx_t^1)\cdots \mu_{k+1}^\epsilon(gx_t^{k+1})\,d\lambda_G(g),
\end{aligned}
\end{equation}
where $(x_t^1,\dots,x_t^{k+1})\in \Phi^{-1}(\vec{t})$.
\vskip.125in

Multiply \eqref{nu1} and \eqref{nu2}, we can write the $L^2$-norm of $\nu^\e$ as
\begin{equation*}
	\begin{aligned}
		\int |\nu^\epsilon(\vec{t})|^2\,d\vec{t}\approx \int &\left(\int_{\Phi^{-1}(\vec{t})} \,\mu_1^\epsilon(x^1)\cdots \mu_{k+1}^\epsilon(x^{k+1})\,\frac{1}{|J_m\Phi(x)|}\,d\mathcal{H}^{d(k+1)-m}(x)\right)\\&\times\left(\int_G \,\mu_1^\epsilon(gx_t^1)\cdots \mu_{k+1}^\epsilon(gx_t^{k+1})\,d\lambda_G(g)\right)\,d\vec{t}\\=\int &\int_{\Phi^{-1}(\vec{t})} \,\mu_1^\epsilon(x^1)\cdots \mu_{k+1}^\epsilon(x^{k+1})\\&\times\left(\int_G \,\mu_1^\epsilon(gx_t^1)\cdots \mu_{k+1}^\epsilon(gx_t^{k+1})\,d\lambda_G(g)\right)\,\frac{1}{|J_m\Phi(x)|}d\mathcal{H}^{d(k+1)-m}(x)\,d\vec{t}.
	\end{aligned}
\end{equation*}

Since the value of this integral does not depend on the choice of $(x_t^1,\dots,x_t^{k+1})\in\Phi^{-1}(\vec{t})$, we can replace $(x_t^1,\dots,x_t^{k+1})$ by $(x^1,\dots,x^{k+1})$ and get
\begin{equation*}
		\int \int_{\Phi^{-1}(\vec{t})} \,\mu_1^\epsilon(x^1)\cdots \mu_{k+1}^\epsilon(x^{k+1})\,\left(\int_G \,\mu_1^\epsilon(gx^1)\cdots \mu_{k+1}^\epsilon(gx^{k+1})\,d\lambda_G(g)\right)\,\frac{1}{|J_m\Phi(x)|}\,d\mathcal{H}^{d(k+1)-m}(x)\,d\vec{t}.
\end{equation*}
By the coarea formula, it equals
\begin{equation*}
	\begin{aligned}
		&\int_{\R^{d(k+1)}} \,\mu_1^\epsilon(x^1)\cdots \mu_{k+1}^\epsilon(x^{k+1})\,\left(\int_G \,\mu_1^\epsilon(gx^1)\cdots \mu_{k+1}^\epsilon(gx^{k+1})\,d\lambda_G(g)\right)dx\\=&\int_{G}\prod_{j=1}^{k+1}\left(\int_{\R^d} \mu_j^\e(x)\,\mu_j^\e(gx)\,dx\right)\,d\lambda_G(g).
	\end{aligned}
\end{equation*}
\section{Frostman measures}
\begin{lemma}
	[Frostman Lemma, see, e.g. \cite{Mat95}]
	Suppose $E\subset\R^d$ and denote $\mathcal{H}^s$ as the $s$-dimensional Hausdorff measure. Then $\mathcal{H}^s(E)>0$ if and only if there exists a probability measure $\mu$ on $E$ such that 
$$\mu(B(x,r))\lesssim r^s $$
	for any $x\in\R^d$, $r>0$.
\end{lemma}

Since by definition $\dH(E)=\sup\{s:\mathcal{H}^s(E)>0\}$, Frostman Lemma implies that for any $\epsilon>0$ there exists a probability measure $\mu_E$ on $E$ such that
\begin{equation}\label{Frostmanmeasure}\mu_E(B(x,r))\lesssim_\epsilon r^{\dH(E)-\epsilon},\ \forall\ x\in\R^d,\ r>0.  
\end{equation}

We need the following properties of Frostman measures throughout this paper.
\begin{lemma}\label{Frostmandecay}
	Suppose $E\subset\R^d$ and $\mu_E$ satisfies \eqref{Frostmanmeasure}, then
	$$\int_{|\xi|\leq R}|\widehat{\mu_E}(\xi)|^2\,d\xi\lesssim_\epsilon R^{d-\dH(E)+\epsilon}. $$
\end{lemma}
\begin{proof}
	Take $\psi\subset C_0^\infty(\R^d)$ whose Fourier transform is positive. Then
\begin{equation*}
\begin{aligned}
	\int_{|\xi|\leq R}|\widehat{\mu_E}(\xi)|^2\,d\xi&\lesssim \int_{|\xi|\leq R}|\widehat{\mu_E}(\xi)|^2\,\widehat{\psi}(\frac{\xi}{R})\,d\xi\\&\leq R^d\iint |\psi(R(x-y))|\,d\mu_E(x)\,d\mu_E(y)\\&\lesssim R^d\int\left(\int_{B(y,R)}\,d\mu_E(x)\right)\,d\mu_E(y)\\&\lesssim_\epsilon R^{d-\dH(E)+\epsilon}.
\end{aligned}	
\end{equation*}
\end{proof}
\begin{theorem}[Wolff (\cite{Wol99}), Erdogan (\cite{Erd05})]
Suppose $E\subset\R^d$, $\dH(E)>\frac{d}{2}$ and $\mu$ satisfies \eqref{Frostmanmeasure}, then
	\begin{equation}\label{sphave}\int_{S^{d-1}}|\widehat{\mu}(R\omega)|^2\,d\omega\lesssim_\epsilon R^{-\frac{d+2\dH(E)-2}{4}+\epsilon}. \end{equation}
\end{theorem}

\section{Proof of Theorem \ref{d>1}}
Denote $E^k=E\times\cdots\times E\subset\R^{kd}$. For $x^i,y^i\in \R^{d}$, $i=1,2,\dots,k$, denote $x=(x^1,\dots,x^k)$, $y=(y^1,\dots,y^k)$ and $$\Phi(x,y)=\prod_{i=1}^k|x^i-y^i|.$$

Notice that $\Phi(x,y)=\Phi(x',y')$ if and only if there exists $z\in\R^{kd}$, $g\in G$ such that $x'=gx+z$, $y'=gy+z$, where 
\begin{equation}\label{dialationgroup}G=\left\{diag(r_1\theta_1,\dots,r_k\theta_k)\in M_{kd\times kd}: \prod_{i=1}^k r_i = 1, \theta_j\in O(d) \right\}. \end{equation}
By Corollary \ref{CorMat}, it suffices to show that
\begin{equation}\label{MatProdSum}\mathcal{M}=\int |\widehat{\mu_{E^k}}(\xi)|^2\left(\int_{G}|\widehat{\mu_{E^k}}(g\xi)|^2\,dg\right)\,d\xi<\infty,\end{equation}
where $\mu_{E^K}=\mu_E\times\cdots\times\mu_E$ and $\mu_E$ is any Frostman measure satisfying \eqref{Frostmanmeasure}.

More precisely,
\begin{equation*}
\mathcal{M}=\int_{\R^d}\cdots\int_{\R^d} \prod_{j=1}^k|\widehat{\mu}(\xi^j)|^2\left(\int_{\prod r_j=1}\prod_{j=1}^k\int_{O(d)}|\widehat{\mu}(r_j\theta\xi^j)|^2\,d\xi\,d\vec{r}\right)d\xi^1\dots d\xi^k.
\end{equation*}

We may assume that $\dH(E)>\frac{d}{2}$, $r_j\approx 1$, $|\xi^j|\approx 2^{m_j}$, $m_k=\min \{m_j\}$. By Wolff-Erdogan's bound \eqref{sphave}, 
$$\int_{O(d)}|\widehat{\mu}(r_k\theta\xi^k)|^2\,d\theta\lesssim_\epsilon 2^{m_k(-\frac{d+2\dH(E)-2}{4}+\epsilon)}. $$
Then $\mathcal{M}$ is bounded above by
\begin{equation*}
\begin{aligned}
	\sum_{m_k}\sum_{m_j> m_k}2^{m_k(-\frac{d+2\dH(E)-2}{4}+\epsilon)} \mathop{\int\cdots\int}\limits_{|\xi^j|\approx 2^{m_j}}\prod_{j=1}^{k} |\widehat{\mu}(\xi^j)|^2 \left(\prod_{j=1}^{k-1}\int\int_{O(d)}|\widehat{\mu}(r_j\theta\xi^j)|^2\,d\theta\,dr_j\right)d\xi^1\dots d\xi^k.
\end{aligned}
\end{equation*}

By polar coordinates and Lemma \ref{Frostmandecay},
\begin{equation*}
	\begin{aligned}
		\prod_{j=1}^{k-1}\int_{r_j\approx 1}\int_{O(d)}|\widehat{\mu}(r_j\theta\xi^j)|^2\,d\theta\,dr_j\lesssim \prod_{j=1}^{k-1}2^{m_j(-d)}\int_{|\eta|\approx 2^{m_j}}|\widehat{\mu}(\eta)|^2\,d\eta\lesssim_\epsilon \prod_{j=1}^{k-1} 2^{m_j(-\dH(E)+\epsilon)}
	\end{aligned}
\end{equation*}

Therefore
\begin{equation*}
	\begin{aligned}\mathcal{M} &\lesssim_\epsilon \sum_{m_k} \sum_{m_j> m_k} 2^{m_k(-\frac{d+2\dH(E)-2}{4}+\epsilon)}\cdot \prod_{j=1}^{k-1}2^{m_j(-\dH(E)+\epsilon)}\prod_{j=1}^k\left( \int_{|\xi^j|\approx 2^{m_j}} |\widehat{\mu}(\xi^j)|^2\,d\xi^j\right)\\&\lesssim_\epsilon \sum_{m_k}\sum_{m_j> m_k}2^{m_k(-\frac{d+2\dH(E)-2}{4}+\epsilon)}\cdot\left(\prod_{j=1}^{k-1} 2^{m_j(-\dH(E)+\epsilon)}\right)\cdot \left(\prod_{j=1}^{k}2^{m_j(d-\dH(E)+\epsilon)}\right)\\&=\ \sum_{m_k}2^{m_k(\frac{3d-6\dH(E)+2}{4}+\epsilon)} \left(\sum_{m_j> m_k}\prod_{j=1}^{k-1}2^{m_j(d-2\dH(E)+\epsilon)}\right).
	\end{aligned}
\end{equation*}

Since $\dH(E)>\frac{d}{2}$, 
$$\sum_{m_j> m_k}\prod_{j=1}^{k-1}2^{m_j(d-2\dH(E)+\epsilon)}\approx 2^{m_k (k-1) (d-2\dH(E)+\epsilon)}.  $$
Hence
$$\mathcal{M}\lesssim_\epsilon \sum_{m_k}2^{m_k(\frac{3d-6\dH(E)+2}{4}+\epsilon)}\cdot2^{m_k k (d-2\dH(E)+\epsilon)}=\sum_{m_k}2^{m_k(\frac{(4k-1)d-(8k-2)\dH(E)+2}{4}+\epsilon)}, $$
which is finite if $\dH(E)>\frac{d}{2}+\frac{1}{4k-1}$, as desired.

\section{Proof of Theorem \ref{sum-prod}}
For any $x=(x_1, x_2)\in\R^2$, denote $x^\perp=(x_2, -x_1)$ and $E^\perp=\{x^\perp:x\in E\}$. Since $E\rightarrow E^\perp$ does not change its Hausdorff dimension, we may work on
$$E\cdot(F+H)^\perp$$
without loss of generality. Notice for all $x,y\neq 0$, $x\cdot y^\perp$ is the (signed) area of the parallelogram generated by $x,y$, and $x\cdot y^\perp=x'\cdot y'^\perp$ if and only if there exists $g\in SL_2(\R)$ such that $x'=gx$, $y'=gy$. Therefore Theorem \ref{main} applies.

With $\Phi(x,y)=x\cdot y^\perp$, $G=SL_2(\R)$, $E_1=E$, $E_2=F+H$, $\mu_1=\mu_E$, $\mu_2=\mu_F*\mu_H$, \eqref{gengenMat} becomes
\begin{equation*}
\begin{aligned}
&\int\left(\int_{\R^d} \mu_1^\e(x)\,\mu_1^\e(gx)\,dx\right)\left(\int_{\R^d} \mu_2^\e(x)\,\mu_2^\e(gx)\,dx\right)\,d\lambda_{SL_2(\R)}(g)\\=&C\int\left(\int_{\R^d} \widehat{\mu_1}(\xi)\widehat{\phi}(\e\xi)\,\overline{\widehat{\mu_1}(g \xi)}\,\overline{\widehat{\phi}(\e g\xi)}\,d\xi\right)\left(\int_{\R^d} \widehat{\mu_2}(\eta)\widehat{\phi}(\e\eta)\,\overline{\widehat{\mu_2}(g \eta)}\,\overline{\widehat{\phi}(\e g\eta)}\,d\eta\right)\,d\lambda_{SL_2(\R)}(g).
\end{aligned}
\end{equation*}
Since $E_1, E_2$ are both compact, we may restrict $\lambda_{SL_2(\R)}$ on a compact subset of $SL_2(\R)$. We first reduces it to the case $|\xi|\approx |\eta|$.

\begin{lemma}\label{fastdecay}
	Suppose $|x|\approx |y|\approx |x\cdot y^\perp|\approx 1$ for any $x\in supp(\mu_1)$, $y\in supp(\mu_2)$. Then for any $\psi\in C_0^\infty (SL_2(\R))$,
	$$\left|\int_{SL_2(\R)} \widehat{\mu_1}(g \xi)\,\widehat{\mu_2}(g \eta)\,\psi(g)d\lambda_{SL_2(\R)}(g)\right|\lesssim_N \max\{|\xi|, |\eta|\}^{-N}$$
	unless $|\xi|^2\approx |\eta|^2\approx |\xi\cdot\eta^\perp|$.
\end{lemma}

We leave the proof to Section \ref{pffastdecay}. Now it suffices to show 
\begin{equation*}
\begin{aligned}
\Lambda_j:= \iint_{|\xi|\approx|\eta|\approx 2^j}&|\widehat{\mu_1}(\xi)|\,|\widehat{\mu_2}(\eta)|\left(\int_{SL_2(\R)} |\widehat{\mu_1}(g \xi)|\,|\widehat{\mu_2}(g \eta)|\,\psi(g) d\lambda_{SL_2(\R)}(g)\right)\,d\xi\,d\eta \\ = \iint_{|\xi|\approx|\eta|\approx 2^j}&|\widehat{\mu_E}(\xi)|\,|\widehat{\mu_F}(\eta)|\,|\widehat{\mu_H}(\eta)|\\&\left(\int_{SL_2(\R)} |\widehat{\mu_E}(g \xi)|\,|\widehat{\mu_F}(g \eta)|\,|\widehat{\mu_{H}}(g \eta)|\,\psi(g) d\lambda_{SL_2(\R)}(g)\right)\,d\xi\,d\eta
\end{aligned}
\end{equation*}
is summable in $j$.
\vskip.125in
We integrate $d\xi$ first. For $g\in supp(\psi)$, $|g\xi|\approx |\xi|$. By Cauchy-Schwartz and Lemma \ref{Frostmandecay},
\begin{equation*}
\begin{aligned}
\int_{|\xi|\approx 2^j} |\widehat{\mu_E}(\xi)|\,|\widehat{\mu_E}(g \xi)|\,d\xi & \leq \left(\int_{|\xi|\approx 2^j}|\widehat{\mu_E}(\xi)|^2\,d\xi\right)^\frac{1}{2}\left(\int_{|\xi|\approx 2^j}|\widehat{\mu_E}(g\xi)|^2\,d\xi\right)^\frac{1}{2}\\ &\approx  \int_{|\xi|\approx 2^j}|\widehat{\mu_E}(\xi)|^2\,d\xi\\&\lesssim_\epsilon  2^{j(2-\dH(E)+\epsilon)}.
\end{aligned}
\end{equation*}
Therefore
$$
\Lambda_j\lesssim_\epsilon 2^{j(2-\dH(E)+\epsilon)}\cdot\int_{|\eta|\approx 2^j}|\widehat{\mu_F}(\eta)|\,|\widehat{\mu_H}(\eta)|\left(\int_{SL_2(\R)} |\widehat{\mu_F}(g \eta)|\,|\widehat{\mu_H}(g \eta)|\,d\lambda_{SL_2(\R)}(g)\right)\,d\eta$$

Then we need the following lemma, whose proof will be given in Section \ref{pfavgsl2}.

\begin{lemma}\label{avgsl2}
	Suppose $E\subset\R^d$ and $\mu_E$ satisfies \eqref{Frostmanmeasure}. Then for any $\psi\in C_0^\infty (SL_2(\R))$, 
$$\int |\widehat{\mu_E}(g\xi)|^2\,\psi(g)d\lambda_{SL_2(\R)}(g)\lesssim_\epsilon |\xi|^{-\dH(E)+\epsilon}. $$
\end{lemma}
\vskip.125in
By Cauchy-Schwartz and Lemma \ref{Frostmandecay}, \ref{avgsl2}, it follows that
\begin{equation*}
	\begin{aligned}
		\Lambda_j&\lesssim_\epsilon\, 2^{j(2-\dH(E)+\epsilon)}\int_{|\eta|\approx 2^j}|\widehat{\mu_F}(\eta)|\,|\widehat{\mu_H}(\eta)|\,\left(\int |\widehat{\mu_F}(g \eta)|^2\,\psi d\lambda(g)\right)^\frac{1}{2}\,\left(\int |\widehat{\mu_H}(g \eta)|^2\,\psi d\lambda(g)\right)^\frac{1}{2}\,d\eta\\&\lesssim_\epsilon \,2^{j(2-\dH(E)+\epsilon)}\,2^{j\frac{-\dH(F)+\epsilon}{2}}\,2^{j\frac{-\dH(H)+\epsilon}{2}}\int_{|\eta|\approx 2^j}|\widehat{\mu_F}(\eta)|\,|\widehat{\mu_H}(\eta)|\,d\eta\\&\lesssim_\epsilon \,2^{j(2-\dH(E)+\epsilon)}\,2^{j\frac{-\dH(F)+\epsilon}{2}}\,2^{j\frac{-\dH(H)+\epsilon}{2}}\left(\int_{|\eta|\approx 2^j}|\widehat{\mu_F}(\eta)|^2\,d\eta\right)^\frac{1}{2}\left(\int_{|\eta|\approx 2^j}|\widehat{\mu_H}(\eta)|^2\,d\eta\right)^\frac{1}{2}\\&\lesssim_\epsilon \,2^{j(2-\dH(E)+\epsilon)}\,2^{j\frac{-\dH(F)+\epsilon}{2}}\,2^{j\frac{-\dH(H)+\epsilon}{2}}\,2^{j\frac{2-\dH(F)+\epsilon}{2}}\,2^{j\frac{2-\dH(H)+\epsilon}{2}}\\&=\,2^{j(4-\dH(E)-\dH(F)-\dH(H)+2\epsilon)},
	\end{aligned}
\end{equation*}
which is summable whenever $\dH(E)+\dH(F)+\dH(H)>4$ and $\epsilon$ is small enough.
\vskip.25in
\section{Proof of Lemma \ref{fastdecay}}\label{pffastdecay}
We need to estimate
\begin{equation*}
\begin{aligned}
&\int_{SL_2(\R)} \widehat{\mu_1}(g \xi)\,\widehat{\mu_2}(g \eta)\,\psi(g) d\lambda_{SL_2(\R)}(g)\\=&\iint\left(\int_{SL_2} e^{-2\pi i(x\cdot g\xi+y\cdot g\eta)}\,\psi(g) d\lambda_{SL_2(\R)}(g)\right)d\mu_1(x)\,d\mu_2(y).
\end{aligned}
\end{equation*}
 
By the Iwasawa decomposition of $SL_2(\R)$ (see, e.g. \cite{Lan85}), $SL_2(\R)=KP$
and 
\begin{equation}\label{decint}
\int_{SL_2(\R)}f(g)d\lambda_{SL_2(\R)}(g)=\int_P\int_K f(kp)d\lambda_K(k)d\lambda_P(p),
\end{equation}
where $K$ is the orthogonal group $O(2)$,
$$P=\left\{\begin{pmatrix}a&b\\0&\frac{1}{a}\end{pmatrix}:a>0,b\in\R \right\}$$ and
$\lambda_K,\lambda_P$ are right Haar measures on $K,P$
respectively. It is also known, by the uniqueness of the Haar measure,
\begin{equation}\label{decmea}
d\lambda_P=da\,db
\end{equation}
 up to a multiplication by a constant.
\vskip.125in
We shall show that if $|\xi|^2\approx|\eta|^2\approx|\xi\cdot\eta^\perp|$ does not hold, then
$$\left|\int_{SL_2} e^{-2\pi i(x\cdot g\xi+y\cdot g\eta)}\,\psi(g) d\lambda_{SL_2(\R)}(g)\right|\lesssim_N \max\{|\xi|,|\eta|\}^{-N}. $$
\vskip.125in
Denote the phase function as
$$\phi=x\cdot g\xi+y\cdot g\eta.$$    
Apply Iwasawa decomposition,
$$x\cdot g\xi=(a\xi_1+b\xi_2)(x_1\cos\theta +x_2\sin\theta)+\frac{\xi_2}{a}(-x_1\sin\theta+x_2\cos\theta),$$
$$y\cdot g\eta=(a\eta_1+b\eta_2)(y_1\cos\theta+y_2\sin\theta)+\frac{\eta_2}{a}(-y_1\sin\theta+y_2\cos\theta).$$

Thus
\begin{equation}\label{nabla}
\begin{aligned}
\phi'_a=&\xi_1(x_1\cos\theta+x_2\sin\theta)-\frac{\xi_2}{a^2}(-x_1\sin\theta+x_2\cos\theta)\\&+\eta_1(y_1\cos\theta-y_2\sin\theta)+\frac{\eta_2}{a^2}(-y_1\sin\theta+y_2\cos\theta),\\
\phi'_b=&\xi_2(x_1\cos\theta
+x_2\sin\theta)+\eta_2(y_1\cos\theta+y_2\sin\theta),\\
\phi'_\theta=&(a\xi_1+b\xi_2)(-x_1\sin\theta +x_2\cos\theta)+\frac{\xi_2}{a}(-x_1\cos\theta-x_2\sin\theta)\\&+(a\eta_1+b\eta_2)(-y_1\sin\theta+y_2\cos\theta)+\frac{\eta_2}{a}(-y_1\cos\theta-y_2\sin\theta).
\end{aligned}
\end{equation}

If $g=id$ is a critical point, then $\nabla\phi$ vanishes at $(a,b,\theta)=(1,0,0)$, i.e.,
\begin{equation}
  \begin{aligned}
    \phi'_a=&x_1\xi_1-x_2\xi_2+y_1\eta_1-y_2\eta_2=0\\\phi'_b=&x_1\xi_2+y_1\eta_2=0\\\phi'_\theta+\phi'_b=&
    x_2\xi_1+y_2\eta_1=0,
  \end{aligned}
\end{equation}
which implies there exists $t\in\R$ such that
$$\xi=-ty^\perp,\ \ \ \eta=tx^\perp.$$ 

Generally, if $g_0\in supp (\psi)$ is a critical point of $\phi$, one can easily see that $g=id$ must be a critical point of $x\cdot g(g_0\xi)+y\cdot g(g_0)\eta$. This means there exists $t\in\R$ such that 
$$g_0\xi=-ty^\perp,\ \ \ g_0\eta=tx^\perp.$$

Since $\psi$ has compact support and $|x|\approx |y|\approx |x\cdot y^\perp|\approx 1$ for any $x\in supp(\mu_1)$, $y\in supp(\mu_2)$, the discussion above shows that $|\nabla\phi|$ could vanish only if $|\xi|^2\approx|\eta|^2\approx|\xi\cdot\eta^\perp|$. In other cases $|\nabla\phi|\gtrsim\max\{|\xi|, |\eta|\}$ and Lemma \ref{fastdecay} follows by integration by parts.
\vskip.25in

\section{Proof of Lemma \ref{avgsl2}}\label{pfavgsl2}
Let $\xi=(\xi_1,\xi_2)$ and first assume $|\xi_2|\geq|\xi_1|.$ 
\vskip.125in
As above, we apply the Iwasawa decomposition of $SL_2(\R)$, $SL_2(\R)=KP$
and 
\begin{equation}\label{decint2}
\int_{SL_2(\R)}f(g)d\lambda_{SL_2(\R)}(g)=\int_K\int_Pf(kp)d\lambda_K(k)d\lambda_P(p),
\end{equation}
where $K$ is the orthogonal $O(2)$, 
$$P=\left\{\begin{pmatrix}a&b\\0&\frac{1}{a}\end{pmatrix}:a>0,b\in\R \right\}$$ and
$$
d\lambda_P=da\,db$$
up to a multiplication by a constant.
\vskip.125in
Topologically $SL_2(\R)$ is homeomorphic to $K\times P$ and $P$ is
homeomorphic to $\R^+\times\R$. So on $supp(\psi)$ we may assume $C^{-1}\leq a\leq C,|b|\leq C$ for some $1<C<\infty$. It follows that
\begin{equation*}
    \int_{SL_2(\R)} \left|\widehat{\mu_E}(g\xi)\right|^2\,\psi(g) d\lambda_{SL_2(\R)}(g)=\int_{O(2)}\int_{-C}^C\int_{C^{-1}}^C \left|\widehat{\mu_E}\left(k\cdot\left(a\xi_1+b\xi_2, \frac{\xi_2}{a}\right) \right)\right|^2da\,db\,d\lambda_{O(2)}(k).
\end{equation*}

Since $K=O(2)$ is compact and $\lambda_K$ is a probability measure, it
suffices to show 
$$ \int_{-C}^C\int_{C^{-1}}^C \left|\widehat{\mu_E}\left(k\cdot \left(a\xi_1,b\xi_1+\frac{\xi_2}{a}\right)\right)\right|^2 da\,db\lesssim|\xi|^{-\dH(E)+\epsilon}.$$
\vskip.125in

Change variablves $u=a\xi_1+b\xi_2$, $v=\frac{\xi_2}{a}$. The Jacobian equals
$$\frac{\partial (u,v)}{\partial(a,b)}= \det\begin{pmatrix} \xi_1&\xi_2\\-\frac{\xi_2}{a^2}&0\end{pmatrix} = \frac{|\xi_2|^2}{a^2}\gtrsim |\xi|^2$$
and therefore
\begin{equation}
\begin{aligned}
  \int_{-C}^C\int_{C^{-1}}^C \left|\widehat{\mu_E}\left(k\cdot \left(a\xi_1+b\xi_2, \frac{\xi_2}{a}\right)\right)\right|^2da\,db&\lesssim\frac{1}{|\xi|^2}\iint_{\{|(u,v)|\lesssim|\xi|\}}|\widehat{\mu_E}(k\cdot(u,v))|^2\,du\, dv\\&= \frac{1}{|\xi|^2}\iint_{\{|(u,v)|\lesssim|\xi|\}}|\widehat{\mu_E}(u,v)|^2\,du\, dv\\&\lesssim|\xi|^{-\dH(E)+\epsilon},
\end{aligned}
\end{equation} 
where the last inequality follows from Lemma \ref{Frostmandecay}.

\vskip.125in
On the other hand, if $|\xi_1|\geq|\xi_2|$, we use another decomposition, $SL_2(\R)=KP'$, where $K=O(2)$ and
$$P'=\left\{\begin{pmatrix}a&0\\b&\frac{1}{a}\end{pmatrix}:a>0,b\in\R \right\}.$$ 
In this case 
$$d\lambda_P=\frac{1}{a^2}\,da\,db$$
and Lemma \ref{avgsl2} follows in a similar way.

\bibliographystyle{abbrv}
\bibliography{/Users/MacPro/Dropbox/Academic/paper/mybibtex.bib}

\end{document}